\documentclass[reqno]{amsart}

\usepackage{amssymb}
\usepackage{amsmath}
\usepackage{amsthm}
\usepackage[sort&compress,numbers]{natbib}
\usepackage[colorlinks,citecolor=blue,linkcolor=blue]{hyperref}

\newtheorem{theorem}{Theorem}

\newtheorem{lemma}{Lemma}

\theoremstyle{definition}
\newtheorem{definition}{Definition}

\theoremstyle{remark}
\newtheorem{remark}{Remark}

\numberwithin{equation}{section}

\renewcommand{\Re}[1]{\operatorname{Re}{#1}}
\renewcommand{\Im}[1]{\operatorname{Im}{#1}}

\begin{document}

\title{Note on Wermuth's theorem on commuting operator exponentials}

\author{Krzysztof Szczygielski}
\address{Institute of Theoretical Physics and Astrophysics, Faculty of Mathematics, Physics and Astrophysics, University of Gdansk, Wita Stwosza 57, 80-308 Gdansk, Poland}

\email{fizksz@ug.edu.pl}

\date{\today}

\begin{abstract}
We apply Wermuth's theorem on commuting operator exponentials to show that if $A, B \in B(X)$, $X$ being Banach space and $A$ of $2\pi i$--congruence free spectrum, then $e^A B = B e^A$ if and only if $AB=BA$. We employ this observation to provide alternative proof of similar result by Chaban and Mortad, applicable for $X$ being a Hilbert space.
\end{abstract}

\maketitle

\section{Introduction. Wermuth's theorem}

It is well known that if two elements $a$, $b$ of noncommutative unital Banach algebra commute, then their exponentials $e^a$ and $e^b$ also commute. Although the converse statement is wrong in general, it was shown by E. Wermuth that if $A$ and $B$ are bounded operators on Banach space satisfying additional condition of being of $2\pi i$--congruence free spectrum, then the opposite implication also remains true.

\begin{definition}
Let $S\subset\mathbb{C}$ and let $z\in\mathbb{C}$ be arbitrarily chosen. We say that $S$ is $z$-congruence free if and only if no two different elements $s_1, s_2 \in S$ exists such that $s_1 = s_2 \, (\mathrm{mod} \, z)$. Equivalently,
\begin{equation}
	\forall \, s_1, s_2 \in S : s_1 - s_2 \neq k z, \qquad k\in\mathbb{Z}\setminus \{0\}.
\end{equation}
\end{definition}

Let $B(X)$ be a Banach algebra of all bounded linear endomorphisms over Banach space $X$. The Wermuth's theorem is formulated for pairs of operators $A,B \in B(X)$ and involves $2\pi i$--congruence freedom of both spectra $\sigma (A)$, $\sigma (B)$ and can be stated as follows:

\begin{theorem}[\textbf{Wermuth}]\label{thm:Wermuth}
Let $A,B\in B(X)$ and let both $\sigma(A),\sigma(B)\subset\mathbb{C}$ to be $2\pi i $--congruence free. Then, $e^A e^B = e^B e^A$ if and only if $AB=BA$.
\end{theorem}

Original formulation \cite{Wermuth1989} of Wermuth's theorem concerned finite dimensional matrix algebra $M_n (\mathbb{C})$ and then was generalized \cite{Wermuth1997} to the case of $B(X)$ for any Banach space $X$. Later on, several results concerning commutativity of exponentials (or their functions) in noncommutative unital algebras emerged (see e.g. \cite{Schmoeger1999,Schmoeger2000,Schmoeger2002,Paliogiannis2003,Bourgeois2007}), with or without making explicit use of $2\pi i$--congruence freedom hypothesis; in particular, interesting result was obtained by Chaban and Mortad in \cite{Chaban2013} for a case of C*-algebra $B(H)$ for Hilbert space $H$, which roughly says that if $A\in B(H)$ is normal operator of ``well-behaved'' spectrum, then $e^A B = Be^A$ if and only if $AB=BA$. In this paper, we formulate a similar result in the Banach space setting in Section \ref{sec:MainResult}. This allows us to we present alternative proof of theorem by Chaban and Mortad in Section \ref{sec:ChabanMortad}.

\section{The main result}
\label{sec:MainResult}

We start with a simple observation of ``shrinking'', or ``rescaling'' property of bounded subsets of complex plane (lemma \ref{lemma:z0congruenceScaling} below). Our main result is then formulated as Theorem \ref{thm:MainResult}.

\begin{lemma}\label{lemma:z0congruenceScaling}
For every bounded nonempty set $U\subset\mathbb{C}$ and every $z \in\mathbb{C} \setminus \{0\}$, there exists such $\tau > 0$ small enough, that set $t U = \{t w : w\in U \}$ is $z$--congruence free for every $t \in (0,\tau)$.
\end{lemma}

\begin{proof}
Let $U \subset \mathbb{C}$ be nonempty and bounded and let
\begin{equation}
	\Delta = \sup_{z_1, z_2 \in U}{|z_1 - z_2|}
\end{equation}
be its \emph{diameter}. If $\Delta > 0$, define
\begin{equation}
	\tau = \frac{|z|}{\Delta}.
\end{equation}
Then, for every $t \in (0, \tau)$ and every pair of complex numbers $z_1, z_2 \in U$ we have $|z_1-z_2| \leqslant \Delta$ which yields
\begin{equation}\label{eq:tZdistance}
	t |z_1-z_2| < \tau \Delta = |z|.
\end{equation}
Let $tU = \{tz : z \in U\}$ and take any $w_1, w_2 \in tU$. As $w_1 = tz_1$ and $w_2 = tz_2$ for some $z_1, z_2 \in U$, equation \eqref{eq:tZdistance} implies 
\begin{equation}
	|w_{1} - w_{2}| = t | z_1 - z_2| < |z|,
\end{equation}
which automatically results in
\begin{equation}
	w_{1} - w_{2} \neq k z, \qquad k \in \mathbb{Z}\setminus \{0 \}
\end{equation}
for every $w_{1}, w_{2} \in t U$, i.e. $w_1 \neq w_2  \, (\mathrm{mod} \, z)$ and $t U$ is $z$-congruence free for any $t \in (0,\tau)$. On the other hand, if $\Delta = 0$, i.e. $U=\{ z_0 \}$ is a singleton, then set $tU$ is automatically $z$--congruence free for any $t\in (0,\infty)$ as $w_1 = w_2 = tz_0$ and
\begin{equation}
	w_{1}-w_{2} = 0 \neq kz, \quad k\in\mathbb{Z}\setminus \{0\}.
\end{equation}
Then, one can take any $\tau \in (0, \infty )$.
\end{proof}

\begin{theorem}\label{thm:MainResult}
Let $A, B \in B(X)$ and let $\sigma(A)$ be $2\pi i$--congruence free. Then, $e^AB=Be^A$ if and only if $AB=BA$.
\end{theorem}

\begin{proof}
We only need to address the ``$\Rightarrow$'' direction. Assume $e^A B = Be^A$. Then, $e^A$ commutes also with every analytic function of $B$, so in particular, for all $t\in\mathbb{R}$,
\begin{equation}
	e^A e^{tB} - e^{tB} e^A = 0.
\end{equation}
As $\sigma (B)$ is a nonempty bounded subset of $\mathbb{C}$, lemma \ref{lemma:z0congruenceScaling} invoked for $U = \sigma(B)$ guarantees that there exists such $\tau > 0$ that $t \sigma (B)$ is $2\pi i$--congruence free for any $t\in(0,\tau )$. Therefore, operator $tB$ is of $2\pi i$--congruence free spectrum. By virtue of Wermuth's result (theorem \ref{thm:Wermuth}),
\begin{equation}
	e^A e^{tB} - e^{tB} e^A = 0 \quad \Rightarrow \quad A \cdot tB - tB \cdot A = t(AB-BA) = 0
\end{equation}
for every $t\in (0,\tau)$, so $A$ and $B$ commute. Necessity is then clear.
\end{proof}

\begin{remark}
The assumption of $2\pi i$--congruence freedom of $\sigma(A)$ cannot be neglected, as the following (canonical) counterexample in $M_2 (\mathbb{C})$ shows: let us choose $A$ and $B$ as, say,
\begin{equation}
	A = \left( \begin{array}{cc} 0 & \pi \\ -\pi & 0 \end{array} \right), \qquad B = \left( \begin{array}{cc} 0 & a \\ a & 0 \end{array} \right)
\end{equation}
for some $a \in \mathbb{R}$. It is easy to verify by direct algebra that $e^A = -I$, so $e^A$ commutes with (any) $B$; however,
\begin{align}
	AB-BA = \left( \begin{array}{cc} 2\pi a & 0 \\ 0 & -2 \pi a \end{array} \right)
\end{align}
which does not vanish unless $a = 0$ and matrices $A$ and $B$ do not commute in general. This is not surprising, as one easily shows $\sigma (A) = \{\pi i, -\pi i\}$ and hence, $\sigma (A)$ is not $2\pi i$--congruence free.
\end{remark}

\section{Relation to result of Chaban and Mortad}
\label{sec:ChabanMortad}

As we mentioned in the Introduction, a similar result concerning C*-algebra $B(H)$, $H$ being a Hilbert space, was obtained some time ago by Chaban and Mortad \cite{Chaban2013}. For any $T \in B(H)$ we define its unique \emph{Cartesian decomposition} of a form
\begin{equation}
	T = \Re{T} + i\,\Im{T},
\end{equation}
where $\Re{T}$ and $\Im{T}$, called the \emph{real part} and the \emph{imaginary part} of $T$, respectively, are self-adjoint and bounded and given by
\begin{equation}
	\Re{T} = \frac{1}{2}(T + T^*), \qquad \Im{T} = \frac{1}{2i}(T-T^*).
\end{equation}
\begin{theorem}[\textbf{Chaban and Mortad}]\label{thm:ChabanMortad}
Let $A,B \in B(H)$ be such that $A$ is normal and $\sigma (\Im{A}) \subset (0,\pi )$, we have
\begin{equation}
	e^A B = B e^A \quad \Longleftrightarrow \quad AB = BA.
\end{equation}
\end{theorem}

This theorem is then proved by means of methods different to ours by employing e.g. Fuglede theorem and without making direct references to $2\pi i$--congruence freedom. However, one can easily show that the $2\pi i$--congruence freedom is in fact the case here which allows to formulate alternative proof of the above result by directly applying theorem \ref{thm:MainResult} (and hence Wermuth's theorem in consequence).

\begin{lemma}\label{lemma:Ais2piiCF}
Let $A\in B(H)$ satisfy assumptions of theorem \ref{thm:ChabanMortad}, i.e. $A$ is normal and $\Im{A} \subset (0,\pi )$. Then, $\sigma(A)$ is $2\pi i$--congruence free.
\end{lemma}

\begin{proof}
If $A$ is normal, i.e. $AA^*=A^* A$, then the Cartesian decomposition of $A$ constitutes of a pair $(\Re{A},\Im{A})$ of commuting normal self-adjoint bounded operators. In such case one can show, applying the Gelfand transform, that spectrum $\sigma(A)$ satisfies
\begin{equation}
	\sigma(A) \subset \sigma(\Re{A}) + i \, \sigma(\Im{A}).
\end{equation}
For convenience, let us enclose $\sigma(A)$ by a rectangle in $\mathbb{C}$. As $\Re{A} = (\Re{A})^{*}$ and $\| \Re{A}\| < \infty$, its spectrum is a bounded subset of $\mathbb{R}$ and one can define
\begin{equation}
	J = \inf{\{[x_1, x_2] \subset \mathbb{R} : \sigma(\Re{A}) \subset [x_1, x_2]\}}
\end{equation}
i.e. $J$ is the smallest interval containing the whole spectrum of $\Re{A}$. Then,
\begin{equation}
	\sigma(A) \subset J + i (0,\pi).
\end{equation}
Take any two $\lambda_1 , \lambda_2 \in J + i(0,\pi)$, $\lambda_{k} = a_k + i b_k $ for $a_k \in J$, $b_k \in (0,\pi )$, $k \in \{1,2\}$; there are two possible cases:

\begin{enumerate}
	\item If it happens that $a_1 = a_2$, then
\begin{equation}
	|\lambda_1 - \lambda_2 | = |b_1 - b_2| \leqslant \sup_{b_1,b_2 \in (0,\pi )}|b_1-b_2| = \pi < 2\pi,
\end{equation}
hence $\lambda_1 - \lambda_2 = i (b_1 - b_2) \neq 2k\pi i$ for every $b_1, b_2 \in (0,\pi)$ and every $k\in\mathbb{Z}\setminus \{0\}$.
	\item If, on the other hand $a_1 \neq a_2$, then automatically
\begin{equation}
	\lambda_1 - \lambda_2 = a_1 - a_2 + i(b_1-b_2) \neq 2 k \pi i, \quad k\in\mathbb{Z}\setminus\{0\}.
\end{equation}
\end{enumerate}
In consequence, set $J + i(0,\pi)$ is $2\pi i$--congruence free; the same can be then stated about $\sigma(\Im{A})$ as its subset.
\end{proof}

Finally, we present an alternative proof of theorem by Chaban and Mortad as a straightforward corollary of the above observation:

\begin{proof}[Proof of theorem \ref{thm:ChabanMortad}]
If $\sigma(\Im{A})\subset (0,\pi)$, then $\sigma(A)$ is $2\pi i$--congruence free by lemma \ref{lemma:Ais2piiCF}. Hence, theorem \ref{thm:MainResult} applies.
\end{proof}

\section{Acknowledgments}

Author acknowledges support received from National Science Centre, Poland via research grant No. 2016/23/D/ST1/02043.

\end{document}